\documentclass[12pt,twoside,a4paper,reqno]{amsart}
\usepackage{bulletinUPB}
\usepackage{graphics}

\usepackage{eucal}
\usepackage{amsmath}
\usepackage{ amsthm, amscd, amsfonts, amssymb, graphicx, color}
\usepackage[small,nohug,heads=vee]{diagrams}
\diagramstyle[labelstyle=\scriptstyle]

\info{A}{xx}{x}{201x}
\title{ WEIGHTED  SEMIGROUP MEASURE ALGEBRA AS A ${\rm WAP}$-ALGEBRA}

\author{H.R. Ebrahimi Vishki}
\address{$^1$Department of Pure Mathematics, Center of Excellence in Analysis on Algebraic Structures (CEAAS), Ferdowsi University of Mashhad, %P.O. Box 1159, Mashhad 91775,
 IRAN, e-mail: {\tt vishki@um.ac.ir}}

\author{B. Khodsiani}
\address{$^2$ Corresponding author, Department of Mathematics, University of Isfahan, Isfahan, IRAN. e-mail: {\tt ${\texttt {b}}_{-}$khodsiani@sci.ui.ac.ir}}

\author{A. Rejali}
\address{$^3$Department of Mathematics, University of Isfahan, Isfahan, IRAN. e-mail: {\tt rejali@sci.ui.ac.ir}}

\begin{document}
\pagestyle{headings}
\maketitle
\begin{abstract}
{\it \quad\qu  A Banach algebra $\A$ for which the natural embedding $x\mapsto \hat{x}$ of $\A$ into $WAP(\A)^*$ is bounded below; that is, for some $m\in \mathbb{R}$ with $m>0$ we have $||\hat{x}||\geq m||x||$, is called a WAP-algebra. Through we mainly concern with weighted measure algebra $M_b(S,\omega),$ where   $\omega$ is a weight on a  semi-topological semigroup $S$. We study those conditions under which  $M_b(S,\omega)$ is a WAP-algebra (respectively  dual Banach algebra). In particular, $M_b(S)$ is a WAP-algebra (respectively  dual Banach algebra) if and  only if $wap(S)$ separates the points of $S$ (respectively $S$  is {\it compactly cancellative semigroup}).   We  apply our results for  improving some older results in the case where $S$ is discrete.}
\end{abstract}
%%%%% Keywords
\begin{Keywords}
{  WAP-algebra, dual Banach
algebra, Arens regularity, weak almost
periodicity.}
\end{Keywords}

%%%%%%% 2000 Mathematics Subject Clasification:
\begin{MSC2010}
43A10, 43A20, 46H15, 46H25.
\end{MSC2010}

%\setcounter{equation}{0}
%%%%%%%%%%%%
%%%%%%%%%%%%%%%%%%%%%%%%%%%%%%%%%%%%%%%%%%%%%%%%%%%%%%%
\section{Introduction and Preliminaries}
Throughout this paper,  we study those conditions under which $M_b(S,\omega)$ is either a ${\rm WAP}$-algebra or a dual Banach algebra.  Our main result in section 2   is that  for a locally compact topological semigroup and a continuous weight $\omega$ on $S$,  the measure algebra $M_b(S,\omega)$ is a dual Banach algebra with respect to  $C_0(S, 1/{\omega})$ if and only if for all compact subsets $F$ and $K$ of $S$ , the maps $\frac{{\displaystyle \chi}_{F^{-1}K}}{\omega}$ and $\frac{{\displaystyle \chi}_{KF^{-1}}}{\omega}$ vanishes at infinity. This  improved the result of   Abolghasemi,  Rejali, \and Ebrahimi Vishki \cite{ARV} to include the case where $S$ is not necessarily discrete. As a consequence in non-weighted case, we conclude for a locally compact topological semigroup $S$, the measure algebra $M_b(S)$ is a dual Banach algebra with respect to  $C_0(S)$ if and only if $S$ is a compactly cancellative semigroup. The later result improved the well known result of   Dales,   Lau and  Strauss \cite[Theorem4.6]{DLS}, $\ell_1(S)$ is dual Banach algebra with respect to $c_0(S)$ if and only if $S$ is weakly cancellative semigroup.

Section 3 is devoted to study  $WAP$-algebras on a semigroup $S$. For every weighted locally compact  semi-topological semigroup $(S,\omega)$, $M_b(S,\omega)$ is a ${\rm WAP}$-algebra if and only if 
 the evaluation map $\epsilon:S\longrightarrow \tilde {X}$ is   one to one, where $\tilde{X}=MM(wap(S,1/\omega))$.
  Our main result  of this section is that $M_b(S)$  is ${\rm WAP}$-algebra if and only if $wap(S)$ separate the points of $S$.  If $C_0(S,1/\omega)\subseteq wap(S,1/\omega)$ then $wap(S,1/\omega)$ separate the points of $S$. Thus $M_b(S,\omega)$ is a ${\rm WAP}$-algebra. We may ask whether, if  $M_b(S,\omega)$  is a ${\rm WAP}$-algebra then $C_0(S,1/\omega)\subseteq wap(S,1/\omega)$. We answer to this question negatively by a counter example. Then we exhibit some necessary and sufficient condition for $c_0(S)\subseteq wap(S)$. we end the paper by some examples which show that our results cannot be improved.

The dual $\A^*$ of a Banach algebra $\A$ can be turned into a Banach $\A-$module in a natural way, by setting

$$\langle  f\cdot a, b\rangle=\langle f ,ab\rangle \ \ \ \ \  {\rm and}    \ \ \ \ \langle a\cdot f, b\rangle=\langle f, ba\rangle\ \ \ \ (a,b\in \A, f\in \A^*).$$

A {\it dual Banach algebra} is a Banach algebra $\A$ such that $\A=(\A_*)^*$, as a Banach space,
for some Banach space $\A_*$, and such that $\A_*$ is a closed $\A-$submodule of $\A^*;$ or equivalently,
the multiplication on $\A$ is separately weak*-continuous. We call $\A_*$ the predual of $\A$. It should be remarked that the predual of a dual Banach algebra need not be unique, in general (see \cite{D, DPW}); so we usually point to the involved predual of a dual Banach algebra.

A functional $f\in \A^*$ is said to be {\it weakly almost periodic} if $\{f\cdot a: \|a\|\leq 1\}$
 is relatively weakly  compact in $\A^*$. We denote by $WAP(\A)$ the set of all weakly
almost periodic elements   of $\A^*.$ It is easy to verify that, $WAP(\A)$ is a (norm) closed subspace of $\A^*$.

It is known that the multiplication of a Banach algebra $\A$ has    two natural but, in general, different extensions (called Arens products) to the second dual $\A^{**}$ each turning $\A^{**}$  into a Banach algebra. When these extensions are equal, $\A$ is said to be (Arens) regular. It can be verified that $\A$ is Arens regular if and only if $WAP(\A)=\A^*$. Further  information  for the Arens regularity of Banach algebras can be found in \cite{D, DL}.

WAP-algebras, as a generalization of the Arens regular algebras, has been introduced and intensively studied in \cite{Da2}. A Banach algebra $\A$ for which the natural embedding $x\mapsto \hat{x}$ of $\A$ into $WAP(\A)^*$  where $\hat{x}(\gamma)=\gamma(x)$ for $\gamma\in WAP(\A)$, is bounded below; that is, for some $m\in \mathbb{R}$ with $m>0$ we have $||\hat{x}||\geq m||x||$,   is called a WAP-algebra.  When $\A$ is Arens regular or dual Banach algebra, the natural embedding of $\A$ into $WAP(\A)^*$  is isometric \cite[Corollary4.6]{Runde}. Also Theorem \ref{mwap} shows that  $M_b(S,\omega)$ is a WAP-algebra if and only if this embedding  is isometric and of course  bounded below, however in general $M_b(S,\omega)$ is  neither Arens regular nor dual Banach algebra.  It has also known that $\A$ is a WAP-algebra if and only if it admits an isometric representation on a reflexive Banach space.

 Moreover,   group algebras are also always WAP-algebras, however;  they are neither dual Banach algebras, nor Arens
regular in the case where the underlying group is not discrete, see \cite{y1}. Ample information about WAP-algebras with further details can be found in the impressive paper \cite{Da2}.

 A character on an ablian algebra $\A$ is a non-zero homomorphism $\tau:\A\rightarrow \mathbb{C}.$ The set of all characters on $\A$ endowed with relative weak$^*$- topology is called character space of $\A$.

Following \cite{BJM}, a semi-topological semigroup is a semigroup $S$  equipped with a Hausdorff  topology under which the multiplication of $S$ is separately continuous. If the multiplication of $S$ is jointly continuous, then $S$ is said to be a topological semigroup. We write $\ell^{\infty}(S)$  for the commutative $C^*$-algebra of all bounded complex-valued   functions on $S$. In the case where $S$ is  locally compact  we also write $C(S)$ and $C_0(S)$ for the $C^*-$subalgebras of $\ell_\infty(S)$ consist of continuous elements and continuous elements  which vanish  at infinity, respectively. We also denote the space of all {\it weakly almost periodic} functions on $S$ by $wap(S)$ which is defined by $$wap(S)=\{f\in C(S): \{R_sf: s\in S\}\  {\rm is\ relatively\ weakly\ compact}\},$$ where $R_sf(t)=f(ts), \ (s,t\in S).$  Then $wap(S)$ is a $C^*-$subalgebra of $C(S)$ and  its character space $S^{wap}$, endowed with the Gelfand topology, enjoys a (Arens type) multiplication that turns it into a  compact  semi-topological semigroup. The evaluation mapping $\epsilon: S\rightarrow S^{wap}$ is  a homomorphism with dense  image and  it induces  an isometric $*-$isomorphism from $C(S^{wap})$ onto $wap(S).$  Many other  properties of $wap(S)$ and its inclusion relations among other function algebras are completely explored in \cite{BJM}.

Let $M_b(S)$ be the Banach space  of all complex regular Borel measures on $S,$ which is known as a Banach algebra with the total variation norm and under the convolution product $*$ defined by the equation
$$\langle \mu*\nu,g\rangle =\int_S\int_Sg(xy)d\mu(x)d\nu(y)\quad (g\in C_0(S))$$
and  as dual of $C_0(S)$.

 Throughout, a weight on $S$ is a Borel measurable   function $\omega:S\rightarrow (0,\infty)$ such that
$$\omega(st)\leq \omega(s)\omega(t),\ \ (s,t\in S).$$
For  $\mu\in M_b(S)$ we define $(\mu\omega)(E)=\int_E\omega d\mu, \ (E\subseteq S\  {\rm is\ Borel\ set}).$ If $\omega\geq 1$, then
 $$M_b(S,\omega)=\{\mu\in M_b(S): \mu\omega\in M_b(S)\}$$
is known as a Banach algebra which is called the {\it weighted semigroup measure algebra} (see \cite{DL,Re1,Re,RV} for further details about such algebras and arbitrary weight functions). Let $S$ be a locally compact semigroup, and let $B(S)$ denote the space of all Borel  measurable and bounded functions on $S$. Set $B(S,1/\omega)=\{f:S\rightarrow \mathbb{C}:f/\omega\in B(S)\}$.  A standard  predual for $M_b(S,\omega)$ is
$$C_0(S,1/\omega)=\{f\in B(S,1/\omega): f/\omega\in C_0(S)\}.$$ Let $f \in C(S,1/\omega)$ then $f$ is called $\omega$-weakly almost  periodic if the set $\{\frac{R_sf}{\omega(s)\omega}: s \in S\}$ is   relatively weakly compact in $C(S)$, where $R_s$  is defined as above. The set of all $\omega$-weakly almost  periodic functions on $S$ is denoted by $wap(S, 1/\omega)$.

In the case where $S$ is discrete we write  $\ell_1(S,\omega)$  instead of  $M_b(S,\omega)$ and $c_0(S,\frac{1}{\omega})$ instead of $C_0(S,\frac{1}{\omega}).$ Then the space
$$\ell_1(S,\omega)=\{f:f=\sum_{s\in S}f(s)\delta_s,\quad ||f||_{1,\omega}=\sum_{s\in S}|f(s)|\omega(s)<\infty\}$$

(where, $\delta_s\in\ell_1(S,\omega) $ be the point mass at $s$ which can be thought as the  vector basis element of $\ell_1(S,\omega)$ ) equipped with the multiplication
$$f*g=\sum_{r\in S}\sum_{st=r}f(s)g(t)\delta_r$$
 (and also define $f*g=0$ if for each $r\in S$ the equation $st=r$ has no solution;) is a Banach algebra which will be called weighted semigroup algebra.
 We also suppress $1$ from the notation whenever $w=1.$

\section{  Semigroup Measure Algebras as Dual Banach Algebras }

 It is  known that the semigroup algebra $\ell_1(S)$ is a dual Banach algebra with respect to $c_0(S)$ if and only if $S$ is weakly cancellative semigroup, see \cite[Theorem4.6]{DLS}. Throughout this section $\omega$ is a continuous weight on $S$. This result has been extended for  the  weighted semigroup algebras $\ell_1(S,\omega)$; \cite{ARV, Da1}. In this section we extend this results to the non-discrete case. We provide some necessary and sufficient conditions that  the measure algebra $M_b(S, \omega)$ becomes  a dual Banach algebra with respect to the  predual $C_0(S,1/\omega)$.

Let $F$ and $K$ be nonempty subsets of a semigroup $S$ and $s\in S$. We put $$s^{-1}F=\{t\in S:st\in F\}\mbox{, and}\quad Fs^{-1}=\{t\in S:ts\in F\}$$
and we also write $s^{-1}t $ for the set $s^{-1}\{t\}$, $FK^{-1}$
for $\cup\{Fs^{-1}:s\in K\}$ and $K^{-1}F$ for
$\cup\{s^{-1}F:s\in K\}$.

{\rm A semigroup $S$ is called left (respectively, right) zero semigroup if $xy=x$
(respectively,  $xy=y$), for all $x,y\in S$.  A semigroup $S$ is called  zero semigroup if there exist  $z\in S$  such that $xy=z$
 for all $x,y\in S$.
 A semigroup $S$ is said to be {\it left (respectively, right) weakly cancellative semigroup} if $s^{-1}F$ (respectively, $Fs^{-1}$) is finite for
each $s\in S$ and each finite subset $F$ of $S$.
 A semigroup $S$ is said to be {\it weakly cancellative semigroup} if it is both left and right
weakly cancellative semigroup.

A semi-topological semigroup $S$ is said to be {\it compactly
cancellative semigroup} if for every compact subsets $F$ and $K$
of $S$ the sets $F^{-1}K $ and $KF^{-1}$ are  compact set.
}
\begin{lemma}\label{ert}{\rm Let $S$ be a  topological semigroup. For every compact subsets $F$ and $K$ of $S$ the sets $F^{-1}K $ and $KF^{-1}$ are closed.
 }\end{lemma}
 \begin{proof}If ${F^{-1}K }$ is empty, then it is closed.  Let $x$ be in the closure of  ${F^{-1}K }$. Then there is a net $(x_\alpha)$ in $F^{-1}K $ such that $x_\alpha \rightarrow x$. Since $x_\alpha\in F^{-1}K $  there is a net $(f_\alpha)$ in $F$ such that $f_\alpha x_\alpha\in K$. Using the  compactness of $F$ and $K$, by passing to a subnet, if necessary,  we may suppose that $f_\alpha x_\alpha\rightarrow k$ and $f_\alpha \rightarrow f$, for some $f\in F$ and $k\in K$. So $fx=k\in K$, that is $x\in {F^{-1}K }$. Therefore $F^{-1}K $ is closed. A similar argument shows that $KF^{-1}$ is also  closed.
\end{proof}
In the next result we study $M_b(S,\omega)$ from the dual Banach algebra point of view.

\begin{theorem}\label{we}{\rm
Let $S$ be  a locally compact topological semigroup and $\omega$
be a continuous weight on $S$. Then the measure algebra $M_b(S,\omega)$ is a
dual Banach algebra with respect to the predual $C_0(S, 1/{\omega})$ if
and only if
for all compact subsets $F$ and $K$ of $S$ , the maps $\frac{{\displaystyle \chi}_{F^{-1}K}}{\omega}$ and $\frac{{\displaystyle \chi}_{KF^{-1}}}{\omega}$ vanishes at infinity.
}\end{theorem}
\begin{proof}Suppose that $M_b(S,\omega)$ is  a dual Banach
algebra with respect to $C_0(S, 1/\omega)$ and let  $K, F$ be nonempty compact subsets of $S$ with a net $(x_\alpha)$ in $F^{-1}K $. Let $C^+_{00}(S)$ denote the non-negative continuous functions with compact support on $S$  and set $C_{00}^+(S,\frac{1}{\omega})=\{f\in C_{0}(S,\frac{1}{\omega}):f/\omega\in C_{00}^+(S)\}$.  Since $\omega$ is continuous we may choose  $f\in C_{00}^+(S,\frac{1}{\omega})$ with $f(K)=1$. There is a net $(t_\alpha)\in F$ such that $t_\alpha x_\alpha\in K$ and the
 compactness of $F$ guaranties the existence of a subnet $(t_\gamma)$ of $(t_\alpha)$ such that $t_\gamma\rightarrow t_0$ for some $t_0$ in $S$.
   Indeed,  for $s\in S$,  $$\lim_{\gamma}(\frac{\delta_{t_\gamma}.f}{\omega})(s)=\lim_{\gamma}\frac{f(t_\gamma s)}{\omega(s)}=
  \frac{f(t_0s)}{\omega(s)}=\frac{\delta_{t_0}.f}{\omega}(s)$$
  there is a $\gamma_0$ such that
$$\bigcup_{\gamma\geq\gamma_0}t_\gamma^{-1}K\subseteq\bigcup_{\gamma\geq\gamma_0}\{s\in
S:(\frac{\delta_{t_\gamma}.f}{\omega})(s)\geq1\}\subseteq\{s\in
S:(\frac{\delta_{t_0}.f}{\omega})(s)\geq\frac{1}{2}\}.$$
 Let $H=\{t_\gamma: \gamma\geq\gamma_0\}\cup\{t_0\}$. Then
$$H^{-1}K =\bigcup_{\gamma\geq\gamma_0} t_\gamma^{-1}K\cup t_0^{-1}K\subseteq
\{s\in S:(\frac{\delta_{t_0}.f}{\omega})(s)\geq\frac{1}{2}\}$$
 and so $H^{-1}K $ is compact. Furthermore, $t_\gamma x_\gamma\in K$, that is
$(x_\gamma)$ is a net in compact set $H^{-1}K$. This means that
$(x_\alpha)$ has a convergent subnet and this  is the proof of necessity.

   The sufficiency  can be  adopted from  \cite[Proposition 3.1]{ARV} with some modifications.  Let $f\in C_0(S, 1/\omega)$ , $\mu\in M_b(S,\omega)
$ and $\varepsilon >0$ be arbitrary. There exist compact
subsets $F$ and $K$ of S such that
 $|\frac{f}{\omega}(s)|<\varepsilon$ for all $s\not\in K$ and $|(\mu\omega)|(S\setminus F)<\varepsilon$.
 \par
Let $s\not\in \{t\in F^{-1}K:\omega(t)\leq \frac{1}{\varepsilon}\}$ , which is compact by hypothesis. Then
\begin{eqnarray*}
|\frac{\mu.f}{\omega}(s)|&=&|\int_S\frac{f(ts)}{\omega(s)}d\mu(t)|\\
&\leq&|\int_F\frac{f(ts)}{\omega(s)}d\mu(t)|+|\int_{S\setminus F}\frac{f(ts)}{\omega(s)}d\mu(t)|\\
&\leq&\int_F|\frac{f(ts)}{\omega(ts)}|\omega(t)d|\mu|(t)+\int_{S\setminus F}|\frac{f(ts)}{\omega(ts)}|\omega(t)d|\mu|(t)\\
&\leq&\varepsilon\int_S\omega(t)d|\mu|(t)+\|f\|_{\omega,\infty}\int_{S\setminus F}\omega(t)d|\mu|(t)\\
&\leq&\varepsilon\|\mu\|_{\omega}+\varepsilon\|f\|_{\omega,\infty}
\end{eqnarray*}
That is, $\mu.f\in C_0(S, 1/\omega)$. Therefore $M_b(S,\omega)$ is a
dual Banach algebra with respect to $C_0(S, 1/\omega)$.\end{proof}
The next Corollaries are immediate consequences of Theorem \ref{we}.
\begin{corollary}\label{w}{\rm Let $S$ be a locally compact topological semigroup. Then the measure algebra $M_b(S)$ is a dual Banach algebra with respect to  $C_0(S)$ if and only if $S$ is a compactly cancellative semigroup.
}\end{corollary}
\begin{corollary}\label{ws}{\cite[Theorem 2.2]{ARV}}{\rm For a semigroup $S$  the semigroup  algebra $\ell_1(S,\omega)$ is a dual Banach algebra with respect to the predual $c_0(S, 1/\omega)$ if and only if for all $s,t\in S$, the maps $\frac{{\displaystyle \chi}_{t^{-1}s}}{\omega}$ and $\frac{{\displaystyle \chi}_{st^{-1}}}{\omega}$ are in $c_0(S)$.
}\end{corollary}
\begin{corollary}\label{yuu}{\rm For a locally compact topological semigroup $S$, if  $M_b(S)$ is a dual Banach algebra with respect to  $C_0(S)$  then $M_b(S,\omega)$ is a dual Banach algebra with respect to  $C_0(S, 1/\omega)$.
}\end{corollary}
\begin{corollary} {\rm Let $S$ be either a left zero (right zero) or a zero locally compact semigroup. There  is a weight $\omega$ such that $M_b(S,\omega)$ is a dual Banach algebra with respect to  $C_0(S,\frac{1}{ \omega})$ if and only if $S$ is $\sigma$-compact.
}\end {corollary}
\begin{proof} Let $K$ and $F$ be compact subsets of $S$. It can be readily verified that in either cases (being left zero, right zero or zero) the sets
$F^{-1}K $ and $KF^{-1}$ are equal to either empty or   $S$. Put $$S_m=\{t\in
F^{-1}K:\omega(t)\leq m\}=\{t\in S: \omega(t)\leq m\}\ \ (m\in\mathbb{N}).$$ Then $S=\cup_{m\in \mathbb{N}}S_m$ and so $S$ is $\sigma$-compact. For the converse let $S=\cup_{n\in \mathbb{N}}S_n$ as a disjoint union of compact sets and let $z$ be a (left or right) zero  for $S$.  Define $\omega(z)=1$ and $\omega(x)=1+n$
for $x\in S_n$ then $\omega $ is a weight on $S$ and
$M_b(S,\omega)$ is a dual Banach algebra. \end{proof}

\begin{examples}\label {Ex}
\begin{enumerate}
\item The set $S=\mathbb{R}^+\times\mathbb{R}$ equipped with the multiplication
 $$(x,y).(x',y')=(x+x',y')\ \ \ ((x,y),(x',y')\in S)$$ and the weight
$\omega(x,y)=e^{-x}(1+|y|)$ is a weighted semigroup. In this example $[a,b]$ denotes a closed interval. As for  $F=[a,b]\times[c,d]$ and
$K=[e,f]\times[g,h],$ with $[c,d]\cap[g,h]\not=\emptyset$
\begin{eqnarray*}
 F^{-1}K &=& \bigcup_{(x,y)\in F}(x,y)^{-1}K \\
  &=& \bigcup_{(x,y)\in F}\{(s,t)\in S:(x,y)(s,t)\in K\} \\
   &=&  \bigcup_{(x,y)\in F}\{(s,t)\in S:(x+s,t)\in K\} \\
   &=&  \bigcup_{(x,y)\in F}[e-x,f-x]\times[g,h]=[e-b,f-a]\times[g,h]
\end{eqnarray*}
and
\begin{eqnarray*}
K  F^{-1}&=& \bigcup_{(x,y)\in F}K(x,y)^{-1} \\
  &=& \bigcup_{(x,y)\in F}\{(s,t)\in S:(s,t)(x,y)\in K\} \\
   &=&  \bigcup_{(x,y)\in F}\{(s,t)\in S:(x+s,y)\in K\} \\
   &=&  \bigcup_{(x,y)\in F}[e-x,f-x]\times\mathbb{R} =[e-b,f-a]\times\mathbb{R}
\end{eqnarray*}
Thus
$$F^{-1}K=[e-b,f-a]\times[g,h]\mbox{\quad and \quad}KF^{-1}= \left\{
\begin{array}{lr}
                                                       [e-b,f-a]\times\mathbb{R} &\mbox{if}\quad [c,d]\cap[g,h]\not=\emptyset \\
                                                       \emptyset &\mbox{if}\quad [c,d]\cap[g,h]=\emptyset
                                                     \end{array}\right.
,$$
$M_b(S)$ is not a dual Banach algebra by Corollary \ref{w}. However,  for all compact subsets $F$ and $K$ of $S$ , the maps $\frac{{\displaystyle \chi}_{F^{-1}K}}{\omega}$ and $\frac{{\displaystyle \chi}_{KF^{-1}}}{\omega}$ vanishes at infinity. So $M_b(S,\omega)$ is a dual Banach algebra with respect to $C_0(S, 1/\omega)$. This shows that  the converse of Corollary \ref{yuu} may not be  valid.
\item For the semigroup  $S=[0,\infty)$  endowed with the zero multiplication,  neither $M_b(S)$ nor $\ell_1(S)$ is a dual Banach algebra. In fact,  $S$ is neither compactly  nor  weakly cancellative semigroup.
%\item Let ${\Bbb R}$ have the ordinary  addition as its semigroup operation, and   with lower limit topology.Then ${\Bbb R}_{\ell}$ is  an Abelian cancellative semigroup
 %,so $\ell_1({\Bbb R}_{\ell})$is dual Banach algebra with respect $c_0({\Bbb R}_{\ell})$ ,but $M({\Bbb R}_{\ell})$ is not a dual Banach algebra with respect  to $C_0({\Bbb R}_{\ell})$.
% Indeed, $K=\{\frac{1}{n}:n\in {\Bbb N}\}\cup\{0\}$ is compact in ${\Bbb R}_{\ell}$ but the closed subset  $K^{-1}\{0\}=-K$ is not.
\end{enumerate}
\end{examples}

\section{ Semigroup Measure Algebras as WAP-Algebras}
  In this section, for a weighted locally compact  semi-topological semigroup $(S,\omega)$, we investigate some necessary  and sufficient condition for $M_b(S,\omega)$ being WAP-algebra. First, we provide some preliminaries.
\begin{definition}{\rm
Let $\tilde{\mathcal{F}}$ be a linear subspace of $B(S,1/\omega)$, and let $\tilde{\mathcal{F}}_r$ denote the set of all real-valued members of $\tilde{\mathcal{F}}$. A mean on $\tilde{\mathcal{F}}$ is a linear functional $\tilde{\mu}$ on $\tilde{\mathcal{F}}$ with the property that
$$\inf_{s\in S} \frac{f}{\omega} ( s ) \leq \tilde{\mu} ( f ) \leq \sup_{s\in S}\frac{f}{\omega}(s)\quad (f\in \tilde{\mathcal{F}}_r  ).$$
The set of all means on $\tilde{\mathcal{F}}$ is denoted by $M(\tilde{\mathcal{F}})$. If $\tilde{\mathcal{F}}$ is also an algebra with the multiplication  given by $f \odot g:= (f. g)/\omega\quad (f, g \in \tilde{\mathcal{F}} )$ and if  $\tilde{\mu}\in M(\tilde{\mathcal{F}})$ satisfies
$$\tilde{\mu}(f \odot g) = \tilde{\mu}(f)\tilde{\mu}(g)\quad (f, g \in \tilde{\mathcal{F}} ),$$
then $\tilde{\mu}$ is said to be multiplicative. The set of all multiplicative means on $\tilde{\mathcal{F}}$ will be denoted by $MM(\tilde{\mathcal{F}})$.

 Let $\tilde{\mathcal{F}}$  be a conjugate closed, linear subspace of $B(S,1/\omega)$ such that $\omega\in \tilde{\mathcal{F}}$.
 \begin{enumerate}
   \item[(i)] For each $s\in S$ define $\epsilon(s)\in M(\tilde{\mathcal{F}})$ by $\epsilon(s)(f)=(f/\omega)(s)\quad (f\in \tilde{\mathcal{F}})$. The mapping
$\epsilon: S\longrightarrow M(\tilde{\mathcal{F}})$ is called the evaluation mapping. If $\tilde{\mathcal{F}}$ is also an algebra, then $\epsilon(S)\subseteq MM(\tilde{\mathcal{F}})$.
   \item[(ii)] Let $\tilde X = M(\tilde{\mathcal{F}})$ (resp. $\tilde X = MM(\tilde{\mathcal{F}})$, if $\tilde{\mathcal{F}}$ is a subalgebra) be endowed with the relative
weak* topology. For each $f \in \tilde{\mathcal{F}}$ the function $\hat{f}\in C(\tilde{X})$ is defined by
$$\hat{f}(\tilde{\mu}):=\tilde{\mu}(f)\quad (\tilde{\mu}\in \tilde{X}).$$
 \end{enumerate}
}\end{definition}

Furthermore, we define $\hat{\tilde{\mathcal{F}}}:=\{\hat{f}:f\in \tilde{\mathcal{F}}\}$
\begin{remark}\begin{enumerate} \item[(i)] The mapping $f\longrightarrow \hat{f}: \tilde{\mathcal{F}}\longrightarrow C(\tilde{X})$ is clearly linear and multiplicative if $\tilde{\mathcal{F}}$
is an algebra and $\tilde{X }= MM(\tilde{\mathcal{F}})$. Also it preserves complex conjugation, and is an
isometry, since for any $f \in \tilde{\mathcal{F}}$
\begin{eqnarray*}
  ||\hat{f}||&=&\sup\{|\hat{f}(\tilde{\mu})|:\tilde{\mu}\in\tilde{X}\}=\sup\{|\tilde{\mu}(f)|:\tilde{\mu}\in\tilde{X}\}\\
  &=&\sup\{|\mu(\frac{f}{\omega})|:\mu\in X\}\leq\sup\{|\mu(\frac{f}{\omega})|:\mu\in C(X)^*, ||\mu||\leq1\}  \\
  &=&||\frac{f}{\omega}||=||f||_{\omega}=\sup\{|\frac{f}{\omega}(s)|:s\in S\}=\sup\{|\epsilon(s)(f)|:s\in S\}\\
  &=& \sup\{|\hat{f}(\epsilon(s))|:s\in S\}\leq||\hat{f}||,
\end{eqnarray*}
 where $X=M(\mathcal{F})$ and $\mathcal{F}=\{ f/\omega:f\in\tilde{\mathcal{F}}\}$. Note that $\hat{f}(\epsilon(s))=\epsilon(s)(f)=(\frac{f}{\omega})(s) (f\in\tilde{F}, s\in S)$. This identity may be written in terms of dual map $\tilde {\epsilon}^*:C(\tilde{X})\longrightarrow C(S,\omega)$ as $\epsilon^*(\hat{f})=f$ for $f\in \tilde {F}$.
\item[(ii)]  Let $\tilde{\mathcal{F}}$ be a conjugate closed linear subspace of $B(S,1/\omega)$, containing $\omega$. Then $M(\tilde{\mathcal{F}})$ is convex and weak* compact, $co(\epsilon(S))$ is weak* dense in $M(\tilde{\mathcal{F}})$, $\tilde{\mathcal{F}}^*$ is the weak* closed
linear span of $\epsilon(S)$, $\epsilon : S \longrightarrow M(\tilde{\mathcal{F}})$ is weak* continuous, and if $\tilde{\mathcal{F}}$ is also an algebra, then
$MM(\tilde{\mathcal{F}})$ is weak* compact and $\epsilon(S)$ is weak* dense in $MM(\tilde{\mathcal{F}})$.
\item[(iii)] Let $\tilde{\mathcal{F}}$  be a $C^*$-subalgebra of $B(S,1/\omega)$, containing $\omega$. If $\tilde {X}$ denotes the space $MM(\tilde{\mathcal{F}})$ with the relative weak* topology, and if $\epsilon : S \longrightarrow \tilde{X}$ denotes the evaluation mapping, then the mapping $f\longrightarrow\hat{f} :\tilde{\mathcal{F}}\longrightarrow C(\tilde {X})$ is an isometric isomorphism with the inverse $\epsilon^*: C(\tilde {X})\longrightarrow\tilde{\mathcal{F}}$.
 \end{enumerate}
 \end{remark}
Let $\tilde{\mathcal{F}}=wap(S,1/\omega)$. Then $\tilde{\mathcal{F}}$ is a $C^*$-algebra and a subspace of $WAP(M_b(S,\omega))$, see \cite[Theorem1.6, Theorem3.3]{Lashkarizadeh}. Set $\tilde{X}=MM(\tilde{\mathcal{F}})$. By the above remark $wap(S, 1/\omega)\cong C(\tilde{X})$ and so $$M_b(\tilde{X})\cong C(\tilde{X})^*\cong wap(S, 1/\omega)^* \subseteq WAP(M_b(S,\omega))^*.$$

 Let  $\epsilon :S\longrightarrow \tilde{X}$  be the evaluation mapping. We also define
$\bar{\epsilon}: M_b(S,\omega)\longrightarrow M_b(\tilde {X})$, by $\langle\bar{\epsilon}(\mu),f\rangle=\int_Sf\omega d\mu$  for $f\in wap(S, 1/\omega)\cong C(\tilde{X})$. Then for every Borel set $B$ in $\tilde {X}$, $$\bar{\epsilon}(\mu)(B)=(\mu\omega)(\epsilon^{-1}(B)).$$
 In particular, $\bar{\epsilon}(\frac{\delta_x}{\omega(x)})=\delta_{\epsilon(x)}$.

 The next theorem is the main result of this section.
\begin{theorem}\label{mwap}{\rm
For every weighted locally compact  semi-topological semigroup
$(S,\omega)$ the following statements
 are equivalent:
 \begin{enumerate}
  \item  The evaluation map $\epsilon:S\longrightarrow \tilde {X}$ is   one to one, where $\tilde{X}=MM(wap(S,1/\omega))$;
  \item $\bar{\epsilon}: M_b(S,\omega)\longrightarrow M_b(\tilde {X})$ is an isometric isomorphism;
\item $M_b(S,\omega)$ is a {\rm WAP}-algebra.
\end{enumerate}
}\end{theorem}
\begin{proof}(1) $\Rightarrow$  (2). Take $\mu\in M_b(S,\omega)$, say $\mu=\mu_1-\mu_2+i(\mu_3-\mu_4)$, where $\mu_j\in M_b(S,\omega)^+$. Set $\nu_j=\bar{\epsilon}(\mu_j)\in M_b(\tilde {X})^+$ for $j=1,2,3,4$, and set
$$\nu=\bar{\epsilon}(\mu)=\nu_1-\nu_2+i(\nu_3-\nu_4).$$
Take $\delta > 0$. For each $j$, there exists Borel set $B_j$ in $\tilde {X}$
 such that $\nu_j(B)\geq 0$  for each Borel subset $B$ of $B_j$ and
 $\sum_{j=1}^4\nu_j(B_j)>||\nu||-\delta$. In fact, by Hahn decomposition theorem for signed measures $\lambda_1=\nu_1-\nu_2$ and $\lambda_2=\nu_3-\nu_4$ there exist  four Borel sets $P_1$, $P_2$, $N_1$ and $N_2$ in $\tilde {X}$ such that
 $$P_1\cup N_1=\tilde {X},\quad P_1\cap N_1=\emptyset,\quad P_2\cup N_2=\tilde {X},\quad P_2\cap N_2=\emptyset$$ and for every Borel set $E$ of $\tilde {X}$ we have,
 $$\nu_1(E)=\lambda_1(P_1\cap E),\ \nu_2(E)=-\lambda_1(N_1\cap E),\ \nu_3(E)=\lambda_2(P_2\cap E),\ \nu_4(E)=-\lambda_2(N_2\cap E).$$ that is $\nu_1,\nu_2,\nu_3,\nu_4$ are  concentrated respectively on $P_1,N_1,P_2,N_2$.

Set $D_1:=P_1\cap N_2,D_2:=N_1\cap P_2, D_3:=P_2\cap P_1, D_4:=N_2\cap N_1$.  Then the family $\{D_1,D_2,D_3,D_4\}$ is a partition of $\tilde {X}$.
 Also for $\delta > 0$ there is a compact set $K$ for which  $$||\nu||-\delta\leq\sum_{j=1}^4||{\nu_j}_{|_{D_j}}||-\delta\leq\sum_{j=1}^4{\nu_j}_{|_{D_j}}(K)
 =\sum_{j=1}^4{\nu_j}(D_j\cap K).$$ Set $B_j=D_j\cap K$. Then  the sets $B_1,B_2,B_3,B_4$ are pairwise disjoint.

Set $C_j=(\epsilon)^{-1}(B_j)$, a Borel set in $S$. Then $(\mu_j\omega)(C_j)=\nu_j(B_j).$
  Since $\epsilon$ is injection, the sets $C_1,C_2,C_3,C_4$ are pairwise disjoint,
and so $$||\mu||_{\omega}\geq \sum_{j=1}^4|\mu\omega(C_j)|\geq \sum_{j=1}^4(\mu_j\omega)(C_j)=\sum_{j=1}^4\nu_j(B_j)>||\nu||-\delta$$
This holds for each $\delta> 0$,  so $||\mu||_{\omega}\geq ||\nu||$.
A similar argument shows that $||\mu||_{\omega}\leq ||\nu||$. Thus $||\mu||_{\omega}= ||\nu||$.

(2)$\Rightarrow$(1). Let $P(S,\omega)$ denote  the subspace  of all probability measures of $M_b(S,\omega)$ and $ext(P(S,\omega))$  the extreme points of unit ball of $P(S,\omega)$. Then $ext(P(S,\omega))=\{\frac{\delta_x}{\omega(x)}:x\in S\}\cong S$ and $ext(P(\tilde {X})\cong \tilde {X}$, see \cite[p.151]{Conway}. By injectivity of $\bar{\epsilon}$, it maps the extreme
points of the unit ball onto the extreme points of the unit ball, thus  $\epsilon:S\longrightarrow \tilde {X}$ is a  one to one map.

   (2)$\Rightarrow$(3). Since $\tilde {X}$ is compact,  $M_b(\tilde {X})$ is  a dual Banach algebra with respect to $C(\tilde {X})$, so it has an isometric representation  $\psi$ on a reflexive Banach space $E$, see \cite{Da2}. In the following commutative diagram,
    \begin{diagram}
 M_b(S,\omega) &\rTo^{\bar{\epsilon}} &M_b(\tilde {X}) \\
&\rdTo^{\phi} &\dTo_{\psi}& \\
& &B(E)
\end{diagram}  If $\bar{\epsilon}$ is isometric, then  so  is  $\phi$.

     Thus $M_b(S,\omega)$ has an isometric representation on a reflexive Banach  space $E$ if  $\bar{\epsilon}$ is an isometric isomorphism.
  So $M_b(S,\omega)$ is a {\rm WAP}-algebra if  $\bar{\epsilon}$ is an isometric isomorphism.

(3)$\Rightarrow$(1). Let $M_b(S,\omega)$ be a {\rm WAP}-algebra. Since $\ell_1(S,\omega)$ is a norm closed subalgebra of $M_b(S,\omega)$, then $\ell_1(S,\omega)$ is a {\rm WAP}-algebra. Using the double limit criterion, it is a simple matter to check that $wap(S,1/\omega)=WAP(\ell_1(S,\omega))$ (see also \cite[Theorem3.7]{Lashkarizadeh}) where we treat $\ell^\infty(S,1/\omega)$ as an $\ell_1(S,\omega)$-bimodule. Then $\bar{\epsilon}:\ell_1(S,\omega)\longrightarrow wap(S,1/\omega)^*$ is  an isometric isomorphism. Since $wap(S,1/\omega)$ is a $C^*$-algebra, as   (2)$\Rightarrow$(1),  $\epsilon:S\longrightarrow \tilde {X}$ is   one to one.
   \end{proof}
   \begin{corollary}{\rm The following statement are equivalent.
   \begin{enumerate}
  \item $\ell_1(S,\omega)$ is a {\rm WAP}-algebra;
\item $M_b(S,\omega)$ is a {\rm WAP}-algebra.
\end{enumerate}
  } \end{corollary}

For $\omega=1$, it is clear that $\tilde{X}=S^{wap}$, and the map  $\epsilon:S\longrightarrow S^{wap}$ is   one to one if and only if $wap(S)$ separates the points of $S$, see \cite{BJM}.
\begin{corollary}{\rm For a locally compact  semi-topological semigroup $S$, the following statements are equivalent:
 \begin{enumerate}
 \item $M_b(S)$ is a {\rm WAP}-algebra;
 \item  $\ell_1(S)$ is a {\rm WAP}-algebra;
 \item The evaluation map $\epsilon:S\longrightarrow S^{wap}$ is   one to one;
 \item $wap(S)$ separates the points of $S$.
  \end{enumerate}
}\end{corollary}
\begin{definition}{\rm Let $X$, $Y$ be sets and $f$ be a complex-valued function on $X\times Y$.
\begin{enumerate}
  \item We say that $f$ is a cluster on $X\times Y$ if for each pair of sequences $(x_n)$, $(y_m)$ of distinct elements of $X,Y$, respectively
\begin{equation}\label{fg}
  \lim_n\lim_mf(x_n,y_m)=\lim_m\lim_nf(x_n,y_m)
\end{equation}
whenever both sides of (\ref{fg}) exist.
  \item If $f$ is cluster and both sides of \ref{fg} are zero (respectively positive) in all cases, we say that $f$ is $0$-cluster(respectively positive cluster).
\end{enumerate}

}\end{definition}
In general $\{f\omega:f\in wap(S)\}\not= wap(S,1/\omega)$.
By using \cite[Lemma1.4]{BR} the following is immediate.
\begin{lemma}\label{ghjk}{\rm
Let $\Omega(x,y)=\frac{\omega(xy)}{\omega(x)\omega(y)}$, for $x,y\in S$. Then
\begin{enumerate}
  \item If $\Omega$ is cluster, then $\{f\omega:f\in wap(S)\}\subseteq wap(S,1/\omega)$;
  \item If $\Omega$ is positive cluster, then $wap(S,1/\omega)=\{f\omega:f\in wap(S)\}$.

\end{enumerate}
}\end{lemma}
It should be noted that if $M_b(S)$ is Arens regular (resp. dual Banach algebra) then $M_b(S,\omega)$ is so.  We don't know that if $M_b(S)$ is {\rm WAP}-algebra, then $M_b(S,\omega)$ is so. The following Lemma give a partial answer to this question.
\begin{corollary}\label{weighted} {\rm Let $S$ be  a locally compact  topological semigroup with a Borel measurable weight function   $\omega$ such that $\Omega$ is cluster on $S\times S$.
 \begin{enumerate}
 \item  If $M_b(S)$ is a {\rm WAP}-algebra, then so is $M_b(S,\omega)$;
 \item If $\ell_1(S)$ is a {\rm WAP}-algebra, then  so is $\ell_1(S,\omega)$.
 \end{enumerate}
}\end{corollary}
\begin{proof}(1) Suppose that $M_b(S)$ is a {\rm WAP}-algebra so $wap(S)$ separates the points of $S$. By lemma\ref{ghjk} for every $f\in wap(S)$, $f\omega\in wap(S,1/\omega)$. Thus the evaluation map $\epsilon:S\longrightarrow \tilde {X}$ is   one to one.

  (2) follows from (1).
\end{proof}
\begin{corollary}\label{homomo} {\rm For a locally compact  semi-topological semigroup $S$,
  \begin{enumerate}
 \item  If $C_0(S)\subseteq wap(S)$, then  the measure algebra $M_b(S)$ is a {\rm WAP}-algebra.
\item  If  $S$ is discrete and $c_0(S)\subseteq wap(S)$, then $\ell_1(S)$ is a {\rm WAP}-algebra.
 \end{enumerate}
}\end{corollary}
\begin{proof}(1)
By \cite[Corollary 4.2.13]{BJM} the map $\epsilon:S\longrightarrow S^{wap}$ is one to one, thus $M_b(S)$ is a {\rm WAP}-algebra.

(2) follows from (1).
\end{proof}

Dales, Lau and Strauss \cite[Theorem 4.6, Proposition 8.3]{DLS} showed that for a semigroup $S$, $\ell^1(S)$ is a dual Banach algebra with respect to $c_0(S)$ if and only if $S$ is weakly cancellative. If $S$ is  left or right weakly cancellative semigroup, then $\ell^1(S)$ is a {\rm WAP}-algebra. The next example shows that the converse is not true, in general.
\begin{example}
Let $S=(\mathbb{N},\min)$ then $wap(S)=c_0(S) \oplus \mathbb{C} $. So $\ell^1(S)$ is a {\rm WAP}-algebra but $S$ is neither left nor right weakly cancellative.
In fact, for $f\in wap(S)$  and all sequences $\{a_n\}$, $\{b_m\}$ with distinct element in $S$, we have $\lim_mf(b_m)=\lim_m\lim_nf(a_nb_m)=\lambda=\lim_n\lim_mf(a_nb_m)=\lim_nf(a_n)$, for some $\lambda\in \mathbb{C}$. This means $f-\lambda\in c_0(S)$ and $wap(S)\subseteq c_0(S) \oplus \mathbb{C} $. The other inclusion is clear.
\end{example}

%In fact, for sequences $\{a_n\}$, $\{b_m\}$ with distinct element in $S$ and  $f\in c_0(S)$, we have %$\lim_nf(a_nb_m)=f(b_m)$. Hence, $\lim_m\lim_nf(a_nb_m)=0$. And Similarly,
 %$\lim_m\lim_nf(a_nb_m)=0=\lim_n\lim_mf(a_nb_m)$. Thus $c_0(S) \oplus \mathbb{C}\subseteq wap(S) $.

 If $\{x_n\}$ and $\{y_m\}$ are sequences in $S$ we obtain an infinite matrix $\{x_ny_m\}$ which has $x_ny_m$ as its entry in the $m$th row and $n$th column.
 As in \cite{BR}, a matrix is said to be of row type $C$ ( resp. column type $C$) if the rows  ( resp. columns ) of the matrix are all constant and distinct. A matrix is of type $C$ if it is constant or of row or column type $C$.

J.W.Baker and A. Rejali in  \cite[Theorem 2.7(v)]{BR} showed that $\ell^1(S)$ is Arens regular if and only if for each pair of sequences $\{x_n\}$, $\{y_m\}$  with distinct elements in $S$ there is a submatrix  of $\{x_ny_m\}$ of type $C$.

A matrix $\{x_ny_m\}$  is said to be upper triangular constant if $x_ny_m=s$ if and only if $m\geq n$ and it is lower triangular constant if  $x_ny_m=s$ if and only if $m\leq n$. A matrix $\{x_ny_m\}$  is said to be  $W$-type if every submtrix of $\{x_ny_m\}$ is  neither upper triangular constant nor lower triangular constant.

\begin{theorem}{\rm
Let $S$ be a semigroup. The following statements are equivalent:
\begin{enumerate}
  \item $c_0(S)\subseteq wap(S)$.%$\ell^1(S)$ is a {\rm WAP}-algebra;
  \item For each $s\in S$ and each pair $\{x_n\}$, $\{y_m\}$ of sequences in $S$,
  $$\{\chi_s(x_ny_m):n<m\}\cap \{\chi_s(x_ny_m):n>m\}\not=\emptyset;$$
\item For  each pair $\{x_n\}$, $\{y_m\}$ of sequences in $S$ with distinct elements, $\{x_ny_m\}$ is a  $W$-type matrix;

\item For every $s\in S$, every infinite set
    $B\subset S$ contains a finite subset $F$ such that $\cap\{sb^{-1}:b\in
    F\}\setminus (\cap\{sb^{-1}: b\in B\setminus F\})$ and $\cap\{b^{-1}s:b\in F\}\setminus (\cap\{b^{-1}s: b\in B\setminus F\})$ are finite.
\end{enumerate}
}\end{theorem}
\begin{proof}(1)$\Leftrightarrow$ (2). For all $s\in S$,  $\chi_s\in wap(S)$ if and only if $$\{\chi_s(x_ny_m):n<m\}\cap \{\chi_s(x_ny_m):n>m\}\not=\emptyset.$$

 (3)$\Rightarrow$ (1)Let $c_0(S)\not\subseteq wap(S)$ then there are sequences $\{x_n\}$, $\{y_m\}$ in $S$ with distinct elements such that for some $s\in S$, $$1=\lim_m\lim_n\chi_s(x_ny_m)\not=\lim_n\lim_m\chi_s(x_ny_m)=0.$$

% Then for all integer $n$ except for a finite number, $\lim_m\chi_s(x_ny_m)=0$. So by omitting a finite number of $y_n$ if necessary, we have: for each $n$, only finitely many terms  $x_ny_m\not=s$ (for $m\in\mathbb{N}$). Likewise, for each fixed $m$ almost all terms $x_my_n=s$ (for $n\in\mathbb{N}$)

 Since $\lim_n\lim_m\chi_s(x_ny_m)=0$,  for $1>\varepsilon>0$ there is a $N\in\mathbb{N}$ such that for all $n\geq N$, $\lim_m\chi_s(x_my_n)<\varepsilon$. This implies for all $n\geq N$, $\lim_m\chi_s(x_my_n)=0$. Then for $n\geq N$, $1>\varepsilon>0$ there is a $M_n\in\mathbb{N}$ such that for all $m\geq M_n$ we have $\chi_s(x_my_n)<\varepsilon$. So if we omit finitely many terms, for all $n\in \mathbb{N}$ there is $M_n\in\mathbb{N}$ such that for all $m\geq M_n$ we have $x_my_n\not=s$. As a similar argument, for all $m\in \mathbb{N}$ there is $N_m\in\mathbb{N}$ such that for all $n\geq N_m$, $x_my_n=s$.

  Let $a_1=x_1$, $b_1$ be the first $y_n$ such that $a_1y_n=s$. Suppose $a_m, b_n$ have been chosen for $1\leq m,n<r$, so that $a_nb_m=s$ if and only if $n\geq m$. Pick $a_r$  to be the first $x_m$ not belonging to the finite set $\cup_{1\leq n\leq r}\{x_m:x_my_n=s\}$. Then $a_rb_n\not=s$ for $n<r$. Pick $b_r$ to be the first $y_n$ belonging to the cofinite set  $\cap_{1\leq n\leq r}\{y_n:x_my_n=s\}$. Then $a_nb_m=s$ if and only if $n\geq m$. The sequences $(a_m)$, $(b_n)$ so constructed satisfy $a_mb_n=s$ if and only if $n\geq m$.
 That is, $\{a_nb_m\}$ is not of $W$-type  and this is a contradiction.

(1)$\Rightarrow$ (3). Let there are sequences $\{x_n\}$, $\{y_m\}$ in $S$ such that $\{x_ny_m\}$ is not a  $W$-type matrix, (say) $x_ny_m=s$ if and only if $m\leq n$. Then $$1=\lim_m\lim_n\chi_s(x_ny_m)\not=\lim_n\lim_m\chi_s(x_ny_m)=0.$$
 So $\chi_s\not\in wap(S)$. Thus $c_0(S)\not\subseteq wap(S)$.

(4)$\Leftrightarrow$ (1)This is Ruppert criterion for $\chi_s\in wap(S)$,  see \cite[Theorem 4]{Ruppert}.

\end{proof}

 We conclude with some examples which show that some of the above results cannot be improved.
 \begin{examples}\ \
 \begin{enumerate}
 \item[(i)]
  Let $S=\mathbb{N}$. Then for $S$ equipped with $\min$ multiplication,  the semigroup algebra $\ell_1(S)$  is a WAP-algebra  but is not neither  Arens regular nor a dual Banach algebra. While, if we replace the $\min$ multiplication  with $\max$ then $\ell_1(S)$ is a dual Banach algebra (so a WAP-algebra) which is not Arens regular. If we change the multiplication of $S$ to the zero multiplication then the resulted semigroup algebra is Arens regular (so a WAP-algebra) which is not a dual Banach algebra. This describes the  interrelation  between the concepts of  being Arens regular  algebra, dual Banach algebra and  WAP-algebra.
\item[(ii)]
Let $S$ be the set of all sequences with $0,1 $ values. We equip
$S$ with coordinate wise multiplication. We denote by $e_n$ the
sequence with all zero unless a $1$ in the $n$-th place. Let
$s=\{x_n\}\in S$, and let $F_w(S)$ be the set of all elements of
$S$ such that $x_i=0$ for only finitely index $i$. It is easy to
see that $F_w(S)$ is countable.  Let $F_w(S)=\{s_1,s_2,\cdots\}$ .
 Recall that, every  element  $g\in \ell^\infty(S)$ can be denoted by
$g=\sum_{s\in S}g(s)\chi_s$, see \cite[p.65]{DL}.  Suppose
$$g=\sum_{s\in S\setminus F_w(S)}g(s)\chi_s$$ be in $wap(S)$, we
show that $g=0$.  Let $s=\{x_n\}\in S$, and $\{k\in
\mathbb{N}:x_k=0\}=\{k_1,k_2,\cdots\}$ be an infinite set. Put
$a_n=s+\sum_{j=1}^n e_{k_j}$ and
$b_m=s+\sum_{i=m}^{\infty}e_{k_i}$. Then $$a_nb_m=\left\{
\begin{array}{lr}
                                                       \sum_{j=m}^ne_{k_j}+s &\mbox{if}\quad m\leq n \\
                                                       s &\mbox{if}\quad m>n
                                                     \end{array}\right.
$$ Thus
$g(s)=\lim_n\lim_m g(a_nb_m)=\lim_m\lim_ng(a_nb_m)=\lim_mg(s+\sum_{i=m}^{\infty}e_{k_i})=0$.

 %$\ell^\infty(S)=\{f:S\rightarrow \mathbb{C}: \|f\|_\infty=\sup\{|f(s)|:s\in S\}<\infty\}$
In fact,
$$wap(S)=\{f\in \ell^\infty(S):f=\sum_{i=1}^\infty f(s_i)\chi_{s_i},\quad s_i \in F_w(S)\}\oplus \mathbb{C}$$

It is clear that $F_w(S)$ is the subsemigroup of $S$ and  $wap(F_w(S))=\ell^\infty(F_w(S))$. So $\ell^1(F_w(S))$ is Arens regular. Let $T$ consists of  those sequences $s=\{x_n\}\in S$ such that $x_i=0$ for  infinitely index $i$, then $T$ is a subsemigroup of $S$ and $wap(T)=\mathbb{C}$. Since $\epsilon_{|_T}: T\longrightarrow S^{wap}$ isn't  one to one, $\ell^1(S)$ is not a  {\rm WAP}-algebra. This shows that in general $\ell^1(S)$  need not be a  {\rm WAP}-algebra.

\item[(iii)]  If we equip   $S=\mathbb{ R}^2$ with the multiplication  $(x,y).(x',y')=(xx',x'y+y')$, then $M_b(S)$ is not a WAP-algebra. Indeed, every non-constant function $f$ over $x$-axis is not in $wap(S)$. Let $f(0,z_1)\not=f(0,z_2)$ and $\{x_m\}$,$\{y_m\}$, $\{\beta_n\}$ be sequences with distinct elements satisfying the recursive equation   \[\beta_nx_m+y_m=\frac{mz_1+nz_2}{m+n}\]
Then
\begin{eqnarray*}
\lim_n\lim_mf((0,\beta_n).(x_m,y_m))&=&\lim_n\lim_mf(0,\beta_nx_m+y_m)\\
&=&\lim_n\lim_mf(0,\frac{mz_1+nz_2}{m+n})\\
&=&f(0,z_1)
\end{eqnarray*}
and similarly
$$\lim_m\lim_nf((0,\beta_n).(x_m,y_m))=f(0,z_2).$$

Thus the map $\epsilon:S\longrightarrow S^{wap}$ isn't  one to one, so $M_b(S)$ is not a WAP-algebra.
This shows that in general $M_b(S)$  need not be a  {\rm WAP}-algebra.

  \item[(iv)]
    Let $S$ be the interval $[\frac{1}{2},1]$ with multiplication $x.y=\max\{\frac{1}{2},xy\}$, where $xy$ is the ordinary multiplication on $\mathbb{R}$. Then for all $s\in S\setminus\{\frac{1}{2}\}$, $x\in S$, $x^{-1}s$ is finite. But $x^{-1}\frac{1}{2}=[\frac{1}{2},\frac{1}{2x}]$. Let $B=[\frac{1}{2},\frac{3}{4})$. Then for all finite subset $F$ of $B$,
    $$\bigcap_{x\in F}x^{-1}\frac{1}{2}\setminus \bigcap_{x\in B\setminus F}x^{-1}\frac{1}{2}=[\frac{2}{3},\frac{1}{2x_F}]$$
 where $x_F=\max F$. By \cite[Theorem 4]{Ruppert}   $\chi_{\frac{1}{2}}\not\in wap(S)$. So $c_0(S\setminus\{\frac{1}{2}\})\oplus \mathbb{C}\subsetneqq wap(S)$. It can be readily verified that $\epsilon:S\longrightarrow S^{wap}$ is   one to one, so $\ell_1(S)$ is a WAP-algebra but $c_0(S)\not\subseteq wap(S)$. This is a counter example for the converse of Corollary \ref{homomo}.

 \item[(v)] Take $T=(\mathbb{N}\cup\{0\},.)$ with 0  as zero
 of $T$ and the multiplication defined by

$$n.m=\left\{
\begin{array}{lr}
                                                     n & \mbox{if}\quad n=m \\
                                                       0 &
                                                       \mbox{otherwise}.
                                                     \end{array}\right.
$$

Then $S=T\times T$ is a semigroup with coordinate wise multiplication.  Now let
$X=\{(k,0): k\in T\}$, $Y=\{(0,k):k\in T\}$ and $Z=X\cup Y$. We
use the Ruppert criterion \cite{Ruppert} to show that
$\chi_z\not\in wap(S)$, for each $z\in Z$. Let $B=\{(k,n): k,n\in
T\}$, then $(k,n)^{-1}(k,0)=\{(k,m):m\not= n\}=B\setminus
\{(k,n)\}$. Thus for all finite subsets $F$ of $B$,
\begin{eqnarray*}
%\nonumber to remove numbering (before each equation)
\left(\cap\{(k,n)^{-1}(k,0): (k,n)\in F\}\right)&\setminus&(
\cap\{(k,n)^{-1}(k,0):
(k,n)\in B\setminus F\})\\
&=& \left(\cap\{(k,0)(k,n)^{-1}: (k,n)\in F\}\right ) \\
&\setminus &(\cap\{(k,n)^{-1}(k,0): (k,n)\in B\setminus F\}) \\
&=& (B\setminus F) \setminus F=B\setminus F
\end{eqnarray*}
and the last set is infinite. This means $\chi_{(k,0)}\not\in wap(S)$.  Similarly $\chi_{(0,k)}\not\in wap(S)$.  Let $f=\sum_{n=0}f(0,n)\chi_{(0,n)}+\sum_{m=1}^\infty f(m,0)\chi_{(m,0)}$ be in $ wap(S)$. For arbitrary fixed $n$ and sequence $\{(n,k)\}$ in $S$, we have $\lim_k f(n,k)=\lim_k\lim_lf(n,l.k)=\lim_l\lim_kf(n,l.k)=f(n,0)$ implies $f(n,0)=0$. Similarly $f(0,n)=0$ and $f(0,0)=0$. Thus $f=0$.   In fact $wap(S)\subseteq \ell^\infty(\mathbb{N}\times\mathbb{N})$. Since $wap(S)$ can not separate the points of $S$ so $\ell_1(S)$ is not a WAP-algebra.
  Let $\omega(n,m)=2^n3^m$ for $(n,m)\in S$. Then $\omega$ is a weight on $S$  such that $\omega\in wap(S,1/\omega)$, so the evaluation map $\epsilon:S\longrightarrow \tilde {X}$ is   one to one. This means $\ell_1(S,\omega)$ is a WAP-algebra but $\ell_1(S)$ is not a WAP-algebra. This is a counter example for the converse of Corollary \ref{weighted}.
\end{enumerate}
\end{examples}

\noindent{{\bf Acknowledgments.}}  This research was supported by
the Center of Excellence for Mathematics at Isfahan university.

\bigskip

%%%%%%%%%%%

\end{document}